\documentclass[12pt]{article} 

\usepackage[margin=1in]{geometry}  

\usepackage{amsmath,amssymb,amsthm}
\usepackage{hyperref} 
\usepackage{multirow} 

 \newcounter{foo}	

\newtheorem{theorem}[foo]{Theorem}
\newtheorem{lemma}[foo]{Lemma}
\newtheorem{definition}[foo]{Definition}

\renewenvironment{proof}{{\noindent\bfseries Proof.}}{$\qed$\\}

\newcommand{\id}[1]{\text{\textnormal{id}}(#1)} 
\newcommand{\rid}[1]{\text{\textnormal{rid}}(#1)}  
\newcommand{\vstar}{{*}}	
\newcommand{\vto}{\leftrightarrow}	

\author{Vahid Fazel-Rezai}
\title{Equivalence Classes of Permutations Modulo Replacements Between 123 and Two-Integer Patterns}

\begin{document}	 
\maketitle

\abstract{We explore a new type of replacement of patterns in permutations, suggested by James Propp, that does not preserve the length of permutations. In particular, we focus on replacements between 123 and a pattern of two integer elements. We apply these replacements in the classical sense; that is, the elements being replaced need not be adjacent in position or value. Given each replacement, the set of all permutations is partitioned into equivalence classes consisting of permutations reachable from one another through a series of bi-directional replacements. We break the eighteen replacements of interest into four categories by the structure of their classes and fully characterize all of their classes.}

\section{Introduction}
\label{sec:introduction}

A permutation is said to contain a pattern if it has a subpermutation order-isomorphic to the pattern. Modern study of permutation patterns was prompted by Donald Knuth in the form of stack-sortable permutations in \cite{knuth68}, but has since evolved into an active combinatorial field. (For various other applications and motivations see Chapters 2 and 3 in the book by Sergey Kitaev \cite{kitaev11}.) Much of the work on permutation patterns has dealt with counting permutations that contain or avoid certain patterns.

The notion of replacing a pattern in a permutation with a new pattern was first mentioned in different forms, as the plactic and Chinese monoids, by Alain Lascoux and Marcel P. Sch\"{u}tzenburger in \cite{LS81}, G\'{e}rard Duchamp and Daniel Krob in \cite{DK92}, and Julien Cassaigne et al.\ in \cite{CEKNH01}. These have since been translated into the language of pattern replacements and further studied. Steven Linton et al.\ consider in \cite{LPRW12} several bi-directional pattern replacements between $123$ and another pattern of length 3, as well as cases where multiple such replacements are allowed at the same time. They inspect these replacements when applied to elements in general position, elements with adjacent positions, and elements with adjacent positions and adjacent values. A couple papers, \cite{PRW11}, by James Propp et al., and \cite{kuszmaul13}, by William Kuszmaul, follow up on replacements on elements with adjacent positions. Together these three papers enumerate the equivalence classes in a general $S_n$ and count the size of the class containing the identity for almost all cases. In addition, William Kuszmaul and Ziling Zhou examine in \cite{KZ13} equivalence classes under more general families of replacements. Throughout all of this work, permutation length was preserved under replacements.

James Propp has suggested considering pattern replacements that do not preserve permutation length; that is, when a pattern is replaced with another pattern of different length. This paper takes the first step in this new direction by examining a group of replacements between patterns with three integer elements and two integer elements. We choose to use the classical type of replacement in which replaced elements need not be adjacent in position or value. To accommodate patterns of different lengths, we use a modified definition of patterns that includes the character $\vstar$ in place of certain integer elements, acting as a placeholder in the replacement procedure. Like previous works, we define equivalence as reachability through a series of bi-directional replacements, using which we can partition the set of permutations of all lengths into equivalence classes.

In particular, in this paper we investigate the equivalence classes for the 18 replacements of the form $123 \vto \beta$, where $\beta$ contains exactly one $\vstar$ and two integers. First, we provide an overview of relevant definitions and notations in Section~\ref{sec:definitions}. Then, we break the 18 replacements into four categories and spend each of Sections~\ref{sec:betaDecreasing}, \ref{sec:dropOnly}, \ref{sec:shiftRightShiftLeft}, and \ref{sec:switchNeighborDrop} dealing with one of these categories. We fully characterize all equivalence classes for each of the considered replacements.

\section{Definitions}
\label{sec:definitions}

We will use the standard definition of a permutation.

\begin{definition}
\label{def:permutation}
A \textbf{permutation} $\pi$ is a finite, possibly empty string consisting of the first $n$ positive integers. We refer to $n$ as the the \textbf{length} of the permutation, denoted by $|\pi|$.
\end{definition}

The permutation of length $0$ is the empty permutation, denoted by $\emptyset$. We will refer to the identity permutation of length $n$, $123\dots n$ (or $\emptyset$ if $n=0$), by $\id{n}$ and the reverse identity permutation of length $n$, $n(n-1)(n-2)\dots 1$ (or $\emptyset$ if $n=0$), by $\rid{n}$. 

We will also mention here a term that will emerge in Section \ref{sec:switchNeighborDrop}:

\begin{definition}
\label{def:leftToRightMinimum}
An element of a permutation is a \textbf{left-to-right minimum} if it has a value less than every element to its left. 
\end{definition}

The terms left-to-right maximum, right-to-left minimum, and right-to-left maximum are defined similarly.

We now introduce the classical notion of a pattern. 

\begin{definition}
\label{def:pattern}
Let $\pi$ and $\mu$ be permutations. A substring $p$ of $\pi$ forms a \textbf{copy} of the \textbf{pattern} $\mu$ if it is order-isomorphic to $\mu$. If such a substring exists, $\pi$ \textbf{contains}~$\mu$. Otherwise, $\pi$ \textbf{avoids}~$\mu$.
\end{definition}

The definition of patterns must be extended for our purposes to accommodate patterns that may contain $\vstar$.

\begin{definition}
\label{def:starPattern}
Let $\rho$ and $\delta$ be strings each consisting of distinct positive integers and $\vstar$'s. A substring $r$ of $\rho$ forms a \textbf{copy} of the \textbf{$\vstar$-pattern} $\delta$ if the following conditions are met:
\begin{itemize}
\item $r$ and $\delta$ have stars in the same positions, and
\item $r$ and $\delta$, when ignoring all stars, are order-isomorphic to one another.
\end{itemize}
If a copy of $\delta$ in $\rho$ exists, $\rho$ \textbf{contains}~$\delta$. Otherwise, $\rho$ \textbf{avoids}~$\delta$.
\end{definition}

We take interest in replacements of patterns in permutations, that are not necessarily adjacent, to form new permutations of possibly different lengths. We define replacements using $\vstar$-patterns to be able to work with changes in length:

\begin{definition}
\label{def:replacement}
Let $\alpha$ and $\beta$ be $\vstar$-patterns of equal length. Given a permutation $\pi$, we say another permutation $\sigma$ is a result of the \textbf{replacement} $\alpha \to \beta$ on $\pi$ if $\sigma$ can be obtained from the following steps on $\pi$:
\begin{enumerate}
\item As necessary, renumber the integers in $\pi$, while preserving relative order, and add instances of $\vstar$ anywhere. Call this result $\rho^{(1)}$.
\item Choose some substring $a$ of $\rho^{(1)}$ that forms a copy of $\alpha$. Also, choose a string $b$ of distinct positive integers and $\vstar$'s such that the following are true:
\begin{itemize}
\item $b$ is itself a copy of $\beta$,
\item all elements common to both $b$ and $\rho^{(1)}$ are contained in $a$, and
\item for all $x\in \mathbb{N}$ contained in both $\alpha$ and $\beta$, if $y\in \mathbb{N}$ in $a$ is at the same position as $x$ in $\alpha$, then $y$ is in $b$ at the same position as $x$ in $\beta$.
\end{itemize}
Replace $a$ in $\rho^{(1)}$ with $b$ and call the result~$\rho^{(2)}$.
\item Drop all instances of $\vstar$ in $\rho^{(2)}$ and renumber, while preserving relative order, so that the final result is a permutation $\rho^{(3)}=\sigma$.
\end{enumerate}
\end{definition}

For example, the intermediate results in applying the replacement $123 \to 3\vstar2$ to the pattern $125$ in $14253$ would be $\rho^{(1)}=14253$, $\rho^{(2)}=54\vstar23$, and $\rho^{(3)}=4312$, respectively. Note that both $\alpha$ and $\beta$ can contain $\vstar$, so, for example, applying $12\vstar \to 3\vstar2$ to $12$ in $14253$ could have intermediate steps $\rho^{(1)} = 1526\vstar3$, $\rho^{(2)} = 45\vstar623$, and $\rho^{(3)}=34512$. Finally, also note that, in some cases, applying a specific replacement to a certain substring of a given permutation can different results depending on the choice of $b$ in the second step.

For clarity, in this paper we will show alongside replacements the involved substrings in the original and resulting permutations with square brackets. In our previous example, this would be~$[125 \to 41]$. Note that in the above definition it is possible that $|\pi| \not= |\sigma|$; that is, replacements do not necessarily preserve length. 

We now introduce the notion of equivalence between permutations of possibly different lengths using two directions of a replacement.

\begin{definition} 
\label{def:equivalent}
We call two permutations $\pi$ and $\sigma$ \textbf{equivalent}, written $\pi \equiv \sigma$, under the bi-directional replacement $\alpha \vto \beta$ if $\sigma$ can be attained through a sequence of $\alpha \to \beta$ and $\beta \to \alpha$ replacements on~$\pi$.
\end{definition}

We use this definition of equivalence to partition the set of all permutations, $S_0 \cup S_1 \cup S_2 \cup \cdots$, into equivalence classes. Our aim is to eventually characterize these classes.

Sometimes we find that it is impossible to apply a given replacement to a permutation, so that it is in its own class, which we will refer to by the following term.

\begin{definition}
\label{def:isolated}
The permutation $\pi$ is \textbf{isolated} under a replacement $\alpha \vto \beta$ if the equivalence class containing it has no other permutations.
\end{definition}

The following property, which arises in particular in Section~\ref{sec:betaDecreasing}, if established gives great insight into the structure of equivalence classes.

\begin{definition}
\label{def:unraveling}
We say a replacement $\alpha \vto \beta$ has the \textbf{unraveling property} if any given permutation is equivalent under $\alpha \vto \beta$ to an identity permutation.
\end{definition}

It is notable that if a replacement has the previous property, then there is at most one class per identity permutation.

It will be helpful in Sections~\ref{sec:dropOnly} and \ref{sec:switchNeighborDrop}
to talk about the shortest permutation equivalent to some given permutation under a replacement, for which we have the following definition.

\begin{definition}
\label{def:primitive}
The \textbf{primitive permutation} $\tau$ of $\pi$ under a replacement $\alpha \vto \beta$ is the unique permutation of shortest length equivalent to $\pi$, if it exists.
\end{definition}

Note that for some permutations and replacements, a shortest equivalent permutation might not be unique, so in such cases we say that a primitive permutation does not exist. 

Finally, we briefly note a symmetry that effectively cuts the number of distinct cases in half: if $\beta$ and $\gamma$ are reverse complements of one another, then $\pi$ and  $\sigma$ are equivalent under $123 \vto \beta$ if and only if their reverse complements are equivalent under $123 \vto \gamma$. (Here reverse means flipped order of elements and complement means flipped value of elements.) For example, because $2314 \equiv 231$ under $123 \vto 13\vstar$, we have $1423 \equiv 312$ under $123 \vto \vstar13$. Note that this symmetry is due to the fact that $123$ is its own reverse complement.

In the remainder of this paper we examine equivalence classes of replacements of the form $123 \vto \beta$, where $\beta$ contains two of $\{ 1,2,3 \}$ and one $\vstar$ in some order. We cover the cases in which the integer elements of $\beta$ are in decreasing order in Section \ref{sec:betaDecreasing}. Then, in Section \ref{sec:dropOnly} we analyze cases where the two integer elements in $\beta$ are in the same positions as in 123. In Section \ref{sec:shiftRightShiftLeft} we consider when the two integer elements of $\beta$ are both shifted left or right one from their positions in 123. We deal with the four remaining cases in Section \ref{sec:switchNeighborDrop}.

\section{$\beta$ Decreasing}
\label{sec:betaDecreasing}

In this section we characterize the classes of the nine replacements in which $\beta$ has integer elements in decreasing order. We will use $123 \vto 31$ to represent an arbitrary replacement out of the three replacements $123 \vto \vstar 31$, $123 \vto 3\vstar1$, and $123 \vto 31\vstar $. Similarly, we use $123 \vto 32 $ to simultaneously discuss all three of $123 \vto \vstar 32 $, $123 \vto 3\vstar 2 $, and $123 \vto 32\vstar $. For analyzing $123 \vto 21$, which denotes $123 \vto \vstar21$, $123 \vto 2\vstar1$, and $123 \vto 21\vstar$, we will make use of reverse complement symmetries with $123 \vto 32$. 

Under all nine replacements, descents are allowed to be rearranged into increasing order, which naturally suggests that they have the unraveling property. This is indeed the case:
\begin{lemma} 
\label{lem:unraveling} 
If $\beta$ is decreasing, then $123 \vto \beta$ has the unraveling property. 
\end{lemma}
\begin{proof}
The following proof is valid for any of $123 \vto 31$ or $123 \vto 32$. This will then cover $123 \vto 21$ by the reverse complement symmetry.

We proceed by inducting on the length of the permutation over the nonnegative integers. For the base case of length zero, we note that the only such permutation, $\emptyset$, is itself an identity permutation. Assume for the inductive step that any permutation of length $n$ is equivalent to some identity permutation and consider any permutation $\pi$ of length $n+1$. By the inductive hypothesis, we may apply replacements on the last $n$ integers of $\pi$ so that they become an increasing string of $m$ integers. Suppose the first element in this result is $k$. Then, we have
\begin{align*}
\pi & \equiv k 123\dots(k-1)(k+1)\dots (m+1) \\
& \equiv 12(k+1)34\dots(k)(k+2)\dots (m+2) \tag*{[$k1 \to 12(k+1)$]} \\
& \equiv 1234(k+2)5\dots(k+1)(k+3)\dots (m+3) \tag*{[$(k+1)3 \to 34(k+2)$]}  \\
& \qquad \vdots \tag{a total of $k-1$ replacements} \\
& \equiv 12345\dots(2k-2)(2k-1)(2k)\dots (m+k),
\end{align*}
or $\pi \equiv 123\dots(m+k) \equiv \id{m+k}$, as desired.
 \end{proof}

Now we turn our attention to only the identity permutations. Here we must deal with $123 \vto 31$ separately:

\begin{lemma} 
\label{lem:id31} 
Under $123\vto 31$, all identity permutations of length 4 or greater are equivalent to one another. 
\end{lemma}
\begin{proof} First, we show that $\id{5} \equiv \id{6}$:
\begin{align*}
12345 & \equiv 2134 \tag*{[$123 \to 21$]}\\
& \equiv 231 \tag*{[$134 \to 31$]}\\
& \equiv 3124 \tag*{[$31 \to 124$]} \\
& \equiv 12435 \tag*{[$31 \to 124$]}\\
& \equiv 123456. \tag*{[$43 \to 345$]}
\end{align*}
Thus, in general for $n \ge 6$ we can apply the above replacements to the first five elements of $\id{n}$ to obtain $\id{n} \equiv \id{n+1}$, so that $\id{5} \equiv \id{6} \equiv \id{7} \equiv \dots$, which was to be shown.

However, this misses $\id{4}$. We now show $\id{4} \equiv \id{7}$, which will complete the proof. Under $123 \vto \vstar31$ (and similarly under $123 \vto 3\vstar1$) we have $1234 \equiv 431 \equiv 321$. Under $123 \vto 31\vstar$, we have $1234 \equiv 421 \equiv 321$. In all three cases, we have $1234 \equiv 321$, from which we continue,
\begin{align*}
1234 & \equiv 321 \\
& \equiv 2341 \tag*{[$32 \to 234$]}\\
& \equiv 23145  \tag*{[$41 \to 145$]}\\
& \equiv 213456  \tag*{[$31 \to 134$]}\\
& \equiv 1234567,  \tag*{[$21 \to 123$]}\\
\end{align*}
as desired.
 \end{proof}

Now we prove the same thing for the other replacements:

\begin{lemma} 
\label{lem:id32and21} 
Under $123\vto 32$ and $123 \vto 21$, all identity permutations of length 4 or greater are equivalent to one another. 
\end{lemma}
\begin{proof} We show this for $123 \vto 32$ and the result for $123 \vto 21$ will follow. 

First, we have that $\id{4} \equiv \id{5}$:
\begin{align*}
1234 & \equiv 132 \tag*{[$234 \to 43$]}\\
& \equiv 2134 \tag*{[$32 \to 134$]} \\
& \equiv 12345. \tag*{[$21 \to 123$]}
\end{align*}
For $n \ge 5$ we can apply the above replacements to the first four elements of $\id{n}$ to obtain $\id{n} \equiv \id{n+1}$, so that $\id{4} \equiv \id{5} \equiv \id{6} \equiv \dots$.
 \end{proof}

Combining the above lemmas, we can explicitly find all the equivalence classes:

\begin{theorem}
\label{thm:betaDecreasingClasses}
If $\beta$ is decreasing, there are only five equivalence classes under $123 \vto \beta$. They are $\{ \emptyset \}, \{ 1 \}$, $\{ 12 \}$, $\{ 123, 21 \}$, and a fifth class containing all other permutations.
\end{theorem}
\begin{proof}
It can easily be verified that each permutation in the first four classes is equivalent to every other permutation in that class, and that applying $123 \vto \beta$ permutation in the first four classes produces a permutation also already in that class. Thus, the first four listed classes contain no other permutations. Also, by Lemma \ref{lem:unraveling} every permutation not in those four classes must be equivalent to an identity of length at least 4. Then by Lemmas \ref{lem:id31} and \ref{lem:id32and21}, all identities of length at least 4 are equivalent, so all remaining permutations are equivalent to one another, forming the fifth class.
\end{proof}

\section{Drop Only: $123 \vto \vstar 23$, $123 \vto 1\vstar 3$, and $123 \vto 12\vstar$}
\label{sec:dropOnly}

The replacements $123 \vto \vstar 23$, $123 \vto 1\vstar 3$, and $123 \vto 12\vstar $  simply drop or add an element in a 123 pattern. In the remainder of this section we will proceed simultaneously with $12\vstar$ and $1\vstar3$ by using $\gamma$ to denote an arbitrary selection from the two, and later use the reverse complement symmetry to state the result of equivalence classes for $123 \vto \vstar23$.

We begin by defining a function that will take any given permutation to what we will show to be its primitive permutation. 

\begin{definition}
\label{def:pFunction}
We define a function $p_\gamma(\pi)$ for a given permutation $\pi$ and replacement $123\vto \gamma$ as follows:
\begin{enumerate}
 \item Begin with the string $\pi$.
 \item If the current string avoids 123, skip to step 4. Otherwise, find the leftmost copy of 123 in the current string first by comparing the smallest elements, then the middle elements (if necessary), and finally the largest elements (if necessary). Apply $123 \to \gamma$ to this copy of 123. 
 \item Repeat step 2 on the resulting string.
 \item Define $p_\gamma(\pi)$ to be the permutation order-isomorphic to the current string.
 \end{enumerate}
\end{definition}

For example, if $\pi = 152364$ and $\gamma = 12\vstar$, the results of the iterations of step 2 are 15234 (using 156), 1524 (using 123), and 152 (using 124), so that $p_\gamma(\pi) = 132$.

The facts below follow immediately from the definition:
\begin{itemize}
\item For every permutation $\pi$, $|p_\gamma(\pi)| \le |\pi|$ with $|p_\gamma(\pi)| = |\pi|$ only if $p_\gamma(\pi)=\pi$.
\item For every permutation $\pi$, $p_\gamma(\pi)$ avoids 123.
\item If a permutation $\pi$ avoids $123$, then $p_\gamma(\pi) = \pi$.
\item For every permutation $\pi$, $p_\gamma(\pi) \equiv \pi$.
\end{itemize}

First we show that $p_\gamma$ is preserved under one direction of the replacement:

\begin{lemma}
\label{lem:singleStep}
If $\sigma$ is the result of $123 \to \gamma$ applied to $\pi$, then $p_\gamma(\pi) = p_\gamma(\sigma)$.
\end{lemma}
\begin{proof}
When written out in terms of their elements, let $\pi = \pi_1 \dots \pi_{k-1} \pi_k \pi_{k+1} \dots \pi_{n}$ and $\sigma = \sigma_1 \dots \sigma_{k-1} \sigma_{k+1} \dots \sigma_n$, so that $n = |\pi| = |\sigma| + 1$ and if $\pi_k$ is dropped from $\pi$ the remaining elements are order-isomorphic to $\sigma$. We now proceed with the proof separately for $\gamma = 12\vstar$ and $\gamma = 1\vstar3$. 

First consider $\gamma = 12\vstar$. We will simultaneously compare the processes of calculating $p_\gamma(\pi)$ and $p_\gamma(\sigma)$. Each iteration of step 2 in Definition \ref{def:pFunction} will be performed on copies of 123 at the same positions for computing $p_\gamma(\pi)$ and $p_\gamma(\sigma)$ when the entire copy of 123 is in the first $k-1$ elements. The first iteration of step 2 in $p_\gamma(\pi)$ for which this is not true must be performed on a copy of 123 in which $\pi_k$ is the third element, because at least one such copy exists (the one on which $123 \vto 12\vstar$ was applied to form $\sigma$). All iterations after this must again be performed on the same positions for $p_\gamma(\pi)$ and $p_\gamma(\sigma)$. Furthermore, the resulting strings of each iteration will be order-isomorphic for the two processes, so that the end results will be equal.

Now suppose $\gamma = 1\vstar3$. Again, each iteration of step 2 performed completely in the first $k-1$ elements will be on the same positions for $p_\gamma(\pi)$ and $p_\gamma(\sigma)$. However, on the iterations for $p_\gamma(\pi)$ in which $\pi_k$ is the third element of a 123 pattern, a copy of 123 will be chosen for $p_\gamma(\sigma)$ in which the first two elements are at the same positions as those for $p_\gamma(\pi)$, but the third element will be to the right of $\sigma_{k-1}$. (We know at least one such third element exists: the third element of the 123 copy on which $123 \vto 1\vstar3$ was applied to form $\sigma$.) Even though the copy of 123 chosen for $p_\gamma(\pi)$ and $p_\gamma(\sigma)$ are different, the middle elements that are dropped will be in the same positions. Finally, on the iteration for $p_\gamma(\pi)$ in which $\pi_k$ is the middle element (this iteration must take place because an appropriate 123 copy must exist), $\pi_k$ will be dropped. The resulting strings at this point for $p_\gamma(\pi)$ and $p_\gamma(\sigma)$ will be order-isomorphic, so the final permutations $p_\gamma(\pi)$ and $p_\gamma(\sigma)$ will be equal.
\end{proof}

Now, we can put a condition on equivalency involving $p_\gamma$:

\begin{lemma}
\label{lem:equivalenceCondition}
Under $123 \vto \gamma$, we have $\pi \equiv \sigma$ if and only if $p_\gamma(\pi) = p_\gamma(\sigma)$.
\end{lemma}
\begin{proof}
First, we prove the if direction. Suppose $p_\gamma(\pi) = p_\gamma(\sigma)$. Then we have $\pi \equiv p_\gamma(\pi) = p_\gamma(\sigma) \equiv \sigma$, so that $\pi \equiv \sigma$ as desired.

For the only if direction assume $\pi \equiv \sigma$. By definition of equivalence, there must exist some sequence of permutations $\pi^{(0)}=\pi, \pi^{(1)}, \pi^{(2)}, \dots, \pi^{(k)}=\sigma$ where $\pi^{(i+1)}$ is the result of performing a $123\to \gamma$ or a $\gamma \to 123$ replacement on $\pi^{(i)}$. We claim that $p_\gamma\left(\pi^{(i)}\right) = p_\gamma\left(\pi^{(i+1)}\right)$ for all $0 \le i \le k-1$. 

Suppose $\pi^{(i+1)}$ is the result of a $123 \to \gamma$ replacement on $\pi^{(i)}$. Then by Lemma \ref{lem:singleStep}, $p_\gamma\left(\pi^{(i)}\right) = p_\gamma\left(\pi^{(i+1)}\right)$. On the other hand, if $\pi^{(i+1)}$ is the result of a $\gamma \to 123$ replacement, then $\pi^{(i)}$ is the result of a $123 \to \gamma$ on $\pi^{(i+1)}$. Thus, by Lemma \ref{lem:singleStep} again we have $p_\gamma\left(\pi^{(i)}\right) = p_\gamma\left(\pi^{(i+1)}\right)$.

Therefore, $p_\gamma(\pi) = p_\gamma\left(\pi^{(0)}\right) = p_\gamma\left(\pi^{(1)}\right) = \dots = p_\gamma\left(\pi^{(k)}\right) = p_\gamma(\sigma)$, as desired. \end{proof}

Now we have enough to show that $p_\gamma(\pi)$ is the primitive permutation of $\pi$:

\begin{lemma}
\label{lem:dropOnlyPrimitive}
Under $123 \vto \gamma$, $p_\gamma(\pi)$ is the primitive permutation of $\pi$.
\end{lemma}
\begin{proof}
By Lemma \ref{lem:equivalenceCondition}, we have $\pi \equiv p_\gamma(\pi)$, so it remains to show that there does not exist a permutation $\sigma$ such that $\sigma \equiv \pi$ with $\sigma$ not order-isomorphic to $p_\gamma(\pi)$ and $|\sigma| \le |p_\gamma(\pi)|$.

For sake of contradiction, assume that some $\sigma \equiv \pi$ exists that is not order-isomorphic to and no longer than $p_\gamma(\pi)$. If $|\sigma|<|p_\gamma(\pi)|$, then $|p_\gamma(\sigma)| \le |\sigma| < |p_\gamma(\pi)|$, so $p_\gamma(\pi) \not= p_\gamma(\sigma) \implies \pi \not \equiv \sigma$, contradiction. Otherwise, $|\sigma|=|p_\gamma(\pi)|$, and we must have $p_\gamma(\pi) = p_\gamma(\sigma)$, so $|\sigma| = |p_\gamma(\sigma)|$. Thus $\sigma$ is order-isomorphic to $p_\gamma(\sigma) = p_\gamma(\pi)$, contradiction.
\end{proof}

We restate the definition of the primitive permutation without using of $p_\gamma(\pi)$ so that we can include $123 \vto \vstar23$.

\begin{theorem}
\label{thm:dropOnlyPrimitive}
For $\beta=12\vstar, 1\vstar3, \vstar23$, the primitive permutation of $\pi$ under $123 \vto \beta$ is the result of repeatedly applying to $\pi$ the replacement $123 \to \beta$ on any choice of a 123 pattern until none exist.
\end{theorem}
\begin{proof}
We first show this for $\beta = 12\vstar$ and $\beta = 1\vstar3$. Suppose the result of applying to $\pi$ the replacement $123 \to \beta$ repeatedly to an arbitrary set of choices of copies of 123 is $\sigma$. If $\sigma = p_\beta(\pi)$, the result is true. Otherwise, because $\sigma$ avoids 123 we have $p_\beta(\sigma) = \sigma \not= p_\beta(\pi)$, so by Lemma \ref{lem:equivalenceCondition} $\pi \not \equiv \sigma$, contradiction.

For $\beta = \vstar23$, we use the reverse complement symmetry: the theorem statement is true for $\beta = 12\vstar$, and the reverse complement of the statement is the statement itself, so it is true for $\beta = \vstar23$.
\end{proof}

As a result, the primitive permutations characterize the equivalence classes:

\begin{theorem}
\label{thm:dropOnlyClasses}
For $\beta=12\vstar, 1\vstar3, \vstar23$, under $123 \vto \beta$, for each $\tau$ avoiding $123$, there exists a distinct class consisting of all $\pi$ whose primitive permutation (as defined in Theorem \ref{thm:dropOnlyPrimitive}) is $\tau$.
\end{theorem}
\begin{proof}
Note that for each $\tau$ avoiding 123, $\tau$ itself along with all other permutations whose primitive permutation is $\tau$ will be equivalent by Definition \ref{def:primitive}, and thus are in the same class.

Suppose now there exists another permutation $\sigma$ that is in the same class as $\pi$, but has primitive permutation $\omega$ different than $\tau$. However, this is a contradiction because both $\omega$ and $\tau$ are defined to be the unique permutation of shortest length equivalent to $\pi$.
\end{proof}

Note that while the above statement of the equivalence classes is the same for all three possible replacements, the classes themselves are different. This is because the primitive permutations can be different for different $\beta$.

\section{Shift Right and Shift Left: $123 \vto \vstar12$ and $123 \vto 23\vstar$}
\label{sec:shiftRightShiftLeft}
We now deal with the replacements that shift two elements of a $123$ pattern to the left or right and drop the third. We may immediately characterize the classes with the following theorem. In the proof, we draw inspiration from the stooge sort, in a manner similar to the proof of Proposition 2.17 in \cite{kuszmaul13}.

\begin{theorem}
\label{thm:shiftRightShiftLeftClasses}
Under $123 \vto \vstar12$ (and similarly under $123 \vto 23\vstar$), each reverse identity is isolated and all other permutations are in the same class.
\end{theorem}
\begin{proof}
Note that the two replacements are reverse complements of one another, and the reverse complement version of the theorem's statement is the same as the statement, so we only work with $123 \vto \vstar 12$.

It is not possible to apply either direction of the replacement $123 \vto \vstar12$ to a reverse identity, so each reverse identity must not be equivalent to any other permutation and is thus isolated.

On the other hand, we claim that the permutations that are not reverse identities are equivalent. Note that immediately we have $12 \equiv 123$. Therefore, for $n \ge 2$ we may transform the first two elements of $\id{n}$ into $123$ so that $\id{n} \equiv \id{n+1}$. Thus, $\id{2} \equiv \id{3} \equiv \id{4} \equiv \dots$.

Now, we will prove that all non-reverse identity permutations of length $n$ are equivalent to $\id{n}$ by inducting on $n \ge 3$. (The cases for $n=0,1,2$ are trivial.) The base case of $n=3$ may be checked computationally. Now, assume the statement is true for $n=k-1$, and suppose $\pi \not= \rid{n}$ is some given permutation of length $n$. If the first $n-1$ elements of $\pi$ are not order-isomorphic to $\rid{n-1}$, we apply the inductive hypothesis to the first $n-1$ elements, then the last $n-1$ elements, and finally the first $n-1$ elements again; the result is $\id{n}$. If the first $n-1$ elements of $\pi$ are order-isomorphic to $\rid{n-1}$, then we instead apply the inductive hypothesis on the last $n-1$ elements, the first $n-1$ elements, and finally the last $n-1$ elements. (We can not have both the first $n-1$ elements and the last $n-1$ elements of $\pi$ order-isomorphic to $\rid{n-1}$, because then $\pi=\rid{n}$.) The result of this procedure is again $\id{n}$, so that $\pi \equiv \id{n}$, as desired.

Because all permutations that are not reverse identities are equivalent to the identity of the same size, and all identities that are not also reverse identities are equivalent, we have that all non-reverse identity permutations are equivalent, completing the proof for $123 \vto \vstar 12$.

Note that taking the reverse complements of each permutation in the classes described above results in exactly the same classes, so the result for $123 \vto 23\vstar$ is the same.
\end{proof}

\section{Switch with Neighbor and Drop: $123 \vto 2\vstar3$, $123 \vto \vstar13$, $123 \vto 13\vstar$, and $123 \vto 1\vstar2$}
\label{sec:switchNeighborDrop}

We first consider $123 \vto 13\vstar$, whose reverse complement is $123 \vto \vstar 13$:

\begin{lemma}
\label{lem:equivalentReplacement}
Two permutations $\pi$ and $\sigma$ of equal length are equivalent under $123 \vto 13\vstar$ if they are equivalent under $123 \vto 132$.
\end{lemma}
\begin{proof}
It suffices to show both directions of $123 \vto 132$ can be performed through a series of $123 \vto 13\vstar$ replacements. Suppose $\pi_1 < \pi_2 < \pi_3$ are three elements of $\pi$ that form a copy of 123. We show that we may transform $\pi_1\pi_2\pi_3$ into $\pi_1\pi_3\pi_2$:
\begin{align*}
\pi &\equiv \dots \pi_1\dots \pi_2  \dots \pi_3  \dots  \\
&\equiv \dots \pi_1\dots \pi_2  \dots  (\pi_2+1)(\pi_3+1) \dots  \tag*{[$\pi_2\pi_3 \to \pi_2(\pi_2+1)(\pi_3+1)$]}\\
 &\equiv \dots \pi_1\dots \pi_3  \dots \pi_2  \dots   \tag*{[$\pi_1\pi_2(\pi_3+1) \to \pi_1\pi_3$]}
\end{align*}
Thus, we may perform $123 \to 132$ replacements by a series of $123 \vto 13\vstar $ replacements. For the other direction, $132 \to 123$, we simply reverse the above process.
\end{proof}

By using this mechanism we may swap any two elements that are not left-to-right minima:

\begin{lemma}
\label{lem:swapNonLRMin}
Suppose a permutation $\pi$ and two of its non-left-to-right minimum elements are given. Order these two elements in decreasing order then drop the rightmost one, and call the result $\pi'$. Under $123\vto 13\vstar$, $\pi \equiv \pi'$.
\end{lemma}
\begin{proof}
Let $\pi = \dots \pi_1 \dots \pi_2 \dots \pi_3 \dots $ and its two non-left-to-right minima be $\pi_2$ and $\pi_3$ (one is not necessarily larger than the other), where $\pi_1$ is the rightmost left-to-right minimum to the left of $\pi_2$. It suffices to show $\pi'$ can be produced with $123\vto 13\vstar$ replacements on $\pi$.

If $\pi_2 < \pi_1$, then $\pi_2$ itself is a left-to-right minimum; therefore, we must have $\pi_2 > \pi_1$. 

Now, if $\pi_3 > \pi_1$, then $\pi_1\pi_2\pi_3$ form a copy of either 123 or 132, and thus $\pi_1$ and $\pi_2$ can be swapped if necessary via Lemma \ref{lem:equivalentReplacement} so that they are in increasing order. Then, applying $123 \to 13\vstar$ produces $\pi'$.

Otherwise, $\pi_3 < \pi_1$. In this case, we look to a fourth element for use in replacements: a left-to-right minimum between $\pi_2$ and $\pi_3$ called $\pi_4$. We can find $\pi_4$ because a left-to-right minimum must exist between $\pi_1$ and $\pi_3$, or else $\pi_3$ is itself a left-to-right minimum. Furthermore, $\pi_4$ must be to the right of $\pi_2$, or else $\pi_1$ was chosen incorrectly. We note that we must have $\pi_4 < \pi_3 < \pi_1 < \pi_2$, so we can perform the following operations. Note that as elements are dropped or added with value less than another element, the latter's value will change by one.
\begin{align*}
\pi &\equiv \dots \pi_1\dots \pi_2  \dots \pi_4  \dots \pi_3  \dots  \\
&\equiv \dots \pi_1\dots \pi_2  \dots \pi_4 \dots \pi_3(\pi_2+1) \dots  \tag*{[$\pi_1\pi_2 \to \pi_1\pi_2(\pi_2+1)$]}\\
&\equiv \dots (\pi_1-1) \dots (\pi_2-1)  \dots \pi_4 \dots \pi_2 \dots  \tag*{[$\pi_4\pi_3(\pi_2+1) \to \pi_4\pi_2$]}\\
&\equiv \dots (\pi_1-1) \dots (\pi_2-1)  \dots \pi_4 \dots  \tag*{[$\pi_1(\pi_2-1)\pi_2 \to \pi_1(\pi_2-1)$]}
\end{align*}
This is indeed $\pi'$.
\end{proof}

While non-left-to-right minima can be manipulated as shown, there are two properties of the set of non-left-to-right minima that must remain unchanged, in addition to the set of left-to-right minima:

\begin{lemma}
\label{lem:invariantProperties}
Under $123 \vto 13\vstar$, two permutations $\pi$ and $\sigma$ are equivalent only if they have the following equal: 
\begin{itemize}
\item the number of left-to-right minima, 
\item the position of the leftmost non-left-to-right minimum, and
\item the largest value (relative to the left-to-right minima) of non-left-to-right minima.
\end{itemize}
\end{lemma}
\begin{proof}
To show that $\pi \equiv \sigma$ only if the they share the three above properties, it suffices to show that $123 \vto 13\vstar$ preserves these properties; moreover, only one direction is necessary to prove because then the reverse must also preserve them. Thus, we consider the $123 \to 13\vstar$ direction applied to the elements $\phi_1 < \phi_2 < \phi_3$ (from left to right) of an arbitrary permutation $\phi$ to produce $\phi'_1$ and $\phi'_3$ in the result $\phi'$ (i.e. $[\phi_1\phi_2\phi_3 \to \phi'_1\phi'_3]$). Also, call $k_1$ the position of $\phi_1$ and $k_2$ that of $\phi_2$ when counting from the left. 

We will undertake the first property by breaking the permutations at the $k_1$-th position. The substrings of the first $k_1$ elements of $\phi$ and $\phi'$ are order-isomorphic, so the two substrings contain the same number of left-to-right minima. On the other hand, the left-to-right minima to the right of $\phi_1$ in $\phi$ and to the right of $\phi'_1$ in $\phi$ have values less than $\phi_1$ and $\phi'_1$. Furthermore, the substring of $\phi$ consisting of all elements except those to the right of and greater than $\phi_1$ is order-isomorphic to the corresponding substring of elements of $\phi'$ that are not both to the right of and greater than $\phi'_1$. Therefore, the numbers of left-to-right minima to the right of the $k_1$-th element in each of $\phi$ and $\phi'$ are the same. We conclude that $123 \to 13\vstar$ preserves the number of left-to-right minima.

To prove the second property we will use the fact that the first $k_2-1$ elements of $\phi$ and $\phi'$ are order-isomorphic. The leftmost non-left-to-right minimum of $\phi$ is at a position of at most $k_2$, because $\phi_2$ is a non-left-to-right minimum. Similarly, the leftmost non-left-to-right minimum of $\phi'$ has position at most $k_2$. If $\phi_2$ is indeed the leftmost non-left-to-right minimum, then $\phi'_3$ will be the leftmost non-left-to-right minimum in $\phi'$ at the same position. Otherwise, the leftmost non-left-to-right minimum is in the first $k_2-1$ elements and thus in the same position in $\phi$ and $\phi'$.

For the third property, it should be noted that when discussing a value relative to those of left-to-right minima we are discussing the number of left-to-right minima less than (or greater than) that value; such a notion is only valid when the number of left-to-right minima is constant (which was shown above). Consider the greatest non-left-to-right minimum in $\phi$ and $\phi'$, which we will call $\phi_i$ and $\phi'_i$ respectively. Note that $\phi_i$ can not be $\phi_2$ because $\phi_3$ is a greater non-left-to-right minimum. Then, the relative order of values of the set of elements from $\phi$ consisting of all left-to-right minima and $\phi_i$ is the same as that of the set of elements from $\phi'$ consisting of all of its left-to-right minima and $\phi'_i$, as desired.
\end{proof}

In terms of three properties we may exactly characterize the equivalence classes:

\begin{theorem}
\label{thm:switchNeighborDropClassesA}
Under $123 \vto 13\vstar$, there exists a distinct equivalence class for every triple of integers $(m, p, v)$, with $1 \le p,v \le m$, consisting of all permutations $\pi$ with the following properties:
\begin{itemize}
\item $\pi$ has $m$ left-to-right minima,
\item the position (from the left) of the leftmost non-left-to-right minimum is $p+1$
\item the value of the largest non-left-to-right minima is less than those of $v$ left-to-right minima
\end{itemize}
In addition, each reverse identity permutation is in a class only containing itself. There are no other classes.
\end{theorem}
\begin{proof}
Note that if $\pi$ is a reverse identity permutation, then it can not undergo either direction of $123 \vto 13\vstar$, so it must be isolated.

For the remainder of the theorem it suffices to show that, given two non-reverse-identity permutations $\pi$ and $\sigma$, $\pi \equiv \sigma$ if and only if they have the same triple $(m,p,v)$. The only if direction was shown in Lemma \ref{lem:invariantProperties}. We will prove the other direction through the use of a primitive permutation.

Suppose both $\pi$ and $\sigma$ have triple $(m,p,v)$. From Lemma \ref{lem:invariantProperties}, any permutation equivalent to $\pi$ must have the same number of left-to-right minima. Also, it must have at least one non-left-to-right minimum. Thus, the shortest permutation equivalent to $\pi$ must have at least $m+1$ elements. In fact, there is exactly one permutation of this length: the permutation of length $m+1$ whose $(p+1)$-th element has value $v+1$ and remaining elements are in decreasing order. We can indeed construct this permutation by applying Lemma \ref{lem:swapNonLRMin} repeatedly to any pair of non-left-to-right minima until the resulting permutation $\tau$ has length $m + 1$. 

In a similar manner, we may construct the primitive permutation of $\sigma$, which must also be $\tau$ because it has the same $(m,p,v)$ triple. Thus, $\pi$ and $\sigma$ have the same primitive permutation and must be equivalent.
\end{proof}

The remaining replacements $123 \vto 1\vstar 2$ and $123\vto 2\vstar3$ (which are reverse complements) have similar equivalence classes. Following logic analogous to Lemma \ref{lem:equivalentReplacement}, Lemma \ref{lem:swapNonLRMin}, Lemma \ref{lem:invariantProperties}, and Theorem \ref{thm:switchNeighborDropClassesA}, we may find that equivalence under $123 \vto 1\vstar2$ implies and is implied by equivalence under $123 \vto 132$, with which we can identify three properties that are necessary and sufficient to infer equivalence. The result classifying the equivalence classes under $123 \vto 1\vstar2$ is stated below:

\begin{theorem}
\label{thm:switchNeighborDropClassesB}
Under $123 \vto 1\vstar 2$, there exists a distinct equivalence class for every triple $(m, p, v)$, with $1 \le p,v \le m$, consisting of all permutations $\pi$ with the following properties:
\begin{itemize}
\item $\pi$ has $m$ left-to-right minima,
\item the value of the smallest non-left-to-right minima is greater than those of $v$ left-to-right minima
\item the position (from the right) of the rightmost non-left-to-right minimum is $p+1$
\end{itemize}
In addition, each reverse identity permutation is isolated. There are no other classes.
\end{theorem}

The results for $123 \vto \vstar13$ and $123 \vto 2\vstar3$ can be found by taking the reverse complement of each statement in Theorems \ref{thm:switchNeighborDropClassesA} and \ref{thm:switchNeighborDropClassesB}, respectively. In particular, all instances of left-to-right minima become right-to-left maxima.

\subsection*{Summary}
\label{subsec:summary}

Table \ref{tab:summary} summarizes the characterization of the equivalence classes for each of the 18 considered replacements. For replacements whose classes are described in the form $\{ \pi \mid \pi \equiv \tau\}$ for a certain $\tau$, refer to their respective sections for algorithms that produce the $\tau$ corresponding to a given $\pi$.

For the sake of abbreviation, we use LR and RL for left-to-right and right-to-left, respectively. In addition, minima and maxima are written min and max, respectively.

\begin{table}[h]
\centering
\begin{tabular}{|c|c|c|c|}
\hline
Category & $\beta$ & \#  Classes & Equivalence Classes \\ \hline \hline
\multirow{9}{70pt}[11pt]{\center{$\beta$ decreasing}} & $\vstar32$ & 5 & $\{ \emptyset \}, \{ 1 \}, \{ 12 \}, \{ 123, 21 \}, \{ \text{all else} \} $ \\
 & $21\vstar$ & 5 & $\{ \emptyset \}, \{ 1 \}, \{ 12 \}, \{ 123, 21 \}, \{ \text{all else} \} $ \\
 & $\vstar31$ & 5 & $\{ \emptyset \}, \{ 1 \}, \{ 12 \}, \{ 123, 21 \}, \{ \text{all else} \} $ \\
 & $2\vstar1$ & 5 & $\{ \emptyset \}, \{ 1 \}, \{ 12 \}, \{ 123, 21 \}, \{ \text{all else} \} $ \\
 & $3\vstar2$ & 5 & $\{ \emptyset \}, \{ 1 \}, \{ 12 \}, \{ 123, 21 \}, \{ \text{all else} \} $ \\
 & $31\vstar$ & 5 & $\{ \emptyset \}, \{ 1 \}, \{ 12 \}, \{ 123, 21 \}, \{ \text{all else} \} $ \\
 & $\vstar21$ & 5 & $\{ \emptyset \}, \{ 1 \}, \{ 12 \}, \{ 123, 21 \}, \{ \text{all else} \} $ \\
 & $3\vstar1$ & 5 & $\{ \emptyset \}, \{ 1 \}, \{ 12 \}, \{ 123, 21 \}, \{ \text{all else} \} $ \\
 & $32\vstar$ & 5 & $\{ \emptyset \}, \{ 1 \}, \{ 12 \}, \{ 123, 21 \}, \{ \text{all else} \} $ \\ \hline
\multirow{3}{70pt}[11pt]{\center{Drop Only}} & $\vstar23$ & $\infty$ & $\{ \pi \mid \pi \equiv \tau \} \mid \tau \text{ avoids 123}$ \\
 & $1\vstar3$ & $\infty$ & $\{ \pi \mid \pi \equiv \tau \} \mid \tau \text{ avoids 123}$ \\
 & $12\vstar$ & $\infty$ & $\{ \pi \mid \pi \equiv \tau \} \mid \tau \text{ avoids 123}$ \\ \hline
\multirow{2}{70pt}[17.5pt]{\center{Shift Right, Shift Left}} & $\vstar12$ & $\infty$ & $\{ \rid{n} \} \mid n \in \mathbb{Z}_{\ge 0},  \{\text{all else}\} $ \\
 & $23\vstar$ & $\infty$ & $\{ \rid{n} \} \mid n \in \mathbb{Z}_{\ge 0},  \{\text{all else}\} $ \\ \hline
\multirow{4}{70pt}[11pt]{\center{Switch with Neighbor and Drop}} & $1\vstar2$ & $\infty$ & $ \{ \pi \mid \pi \equiv \tau\} \mid \tau \text{ has 0 or 1 non-LR min}$ \\
 & $13\vstar$ & $\infty$ & $ \{ \pi \mid \pi \equiv \tau\} \mid \tau \text{ has 0 or 1 non-LR min}$ \\
 & $\vstar13$ & $\infty$ & $ \{ \pi \mid \pi \equiv \tau\} \mid \tau \text{ has 0 or 1 non-RL max}$ \\
 & $2\vstar3$ & $\infty$ & $ \{ \pi \mid \pi \equiv \tau\} \mid \tau \text{ has 0 or 1 non-RL max}$ \\ \hline
\end{tabular}

\caption{Summary of Classes Under Replacements of the Form $123 \vto \beta$}
\label{tab:summary}

\end{table}

\subsection*{Acknowledgments}
\label{subsec:acknowledgements}

The author would like to thank his mentor Dr.\ Tanya Khovanova for invaluable guidance and greatly appreciates the time she has invested in the project. He would also like to thank Prof.\ James Propp for originally suggesting the project and providing useful feedback on the topic of the paper. Finally, the author is grateful to MIT PRIMES for their support and arranging the conditions for this research to take place.

\end{document}